\documentclass[12pt]{article}
\usepackage{amsmath, amssymb, amsthm, a4wide}
\usepackage{geometry}
\usepackage{fancyhdr}
\usepackage{enumerate}
\usepackage{multirow}
\newtheorem{theorem}{Theorem}[section]

\newtheorem{corollary}[theorem]{Corollary}
\newtheorem{definition}[theorem]{Definition}
\newtheorem{proposition}[theorem]{Proposition}
\newtheorem{example}[theorem]{Example}
\newtheorem{remark}[theorem]{Remark}

\numberwithin{equation}{section}

\allowdisplaybreaks[4]
\newcounter{newlist}

\newcounter{nnnewlist}

\newcounter{nelist}

\def\tX{\tilde{X}}

\def\lt{\left}
\def\rt{\right}

\def\lu{\underline{\mu}}
\def\ou{\overline{\mu}}
\def\ls{\underline{\sigma}}
\def\os{\overline{\sigma}}

\def\eg{{\mathbb{E}_G}}
\def\vp{\varphi}
\def\ve{\varepsilon}

\def\br{\mathbb{R}}

\def\cp{\mathcal{P}}
\def\cf{\mathcal{F}}
\def\cb{\mathcal{B}}

\def\cf{\mathcal{F}}

\def\lu{\underline{\mu}}
\def\ou{\overline{\mu}}
\def\ls{\underline{\sigma}}
\def\os{\overline{\sigma}}

\def\be{\hat{\mathbb{E}}}

\def\bn{\mathbb{N}}

\def\pe{{\mathbb{E}^{\cp}}}

\def\eg{\mathbb{E}_G}

\def\vp{\varphi}
\def\ve{\varepsilon}

\begin{document}

\title{On the functional central limit theorem with mean-uncertainty}

\author{Xinpeng LI\footnotemark[2]}
 \renewcommand{\thefootnote}{\fnsymbol{footnote}}
 \footnotetext[2]{Research Center for Mathematics and Interdisciplinary Sciences; Frontiers Science Center for Nonlinear Expectations (Ministry of Education), Shandong University, 266237, Qingdao, China; School of Mathematics, Shandong University, 250100, Jinan, China.\newline
 Email: lixinpeng@sdu.edu.cn (Xinpeng LI)}
\date{ }

\maketitle

\noindent\textbf{Abstract.}
We introduce a new basic model for independent and identical distributed sequence on the canonical space $(\mathbb{R}^\bn,\mathcal{B}(\br^\bn))$ via probability kernels with model uncertainty. Thanks to the well-defined upper and lower variances, we obtain a new functional central limit theorem with mean-uncertainty by the means of martingale central limit theorem and stability of stochastic integral in the classical probability theory. Then we extend it from the canonical space to the general sublinear expectation space. The corresponding proofs are purely probabilistic and do not rely on the nonlinear partial differential equation.\\



\section{Introduction}

The notions of independence and identical distribution are very important in the probability and statistics. In classical probability theory, the familiar way to construct canonical sequence $\{X_i\}_{i\in\bn}$ of independent random variables which having prescribed marginal laws $\{\mu_i\}_{i\in\bn}$ is by product measure $\mu=\otimes_{i=1}^\infty\mu_i$ on the canonical space $(\br^\bn,\cb(\br^\bn))$ with $X_i(\omega)=\omega_i, \omega=(\omega_1,\cdots,\omega_n,\cdots)\in\br^\bn$. In general, any joint law $\mu$ on $(\br^\bn,\cb(\br^\bn))$ can be decomposed in terms of its probability kernels in the form (see Yan \cite{yan}): $\forall n\in\bn$, $ \forall A\in\cb(\br^n)$,
$$\mu(A\times\br^{\bn-n})=\int_\br\mu_1(dx_1)\int_{\br}\kappa_{2}(x_1,dx_2)\cdots\int_\br I_{A}\kappa_n(x_1,\cdots,x_{n-1},dx_{n}).$$
In particular, if the probability kernel $\kappa_i(x_1,\cdots,x_{i-1},\cdot)$ is independent of $(x_1,\cdots,x_{i-1})$ for all $i\geq 2$, the canonical random variables $\{X_i\}_{i\in\bn}$ are independent. If further assume that $\kappa_i(x_1,\cdots,x_{i-1},\cdot)=\mu_1(\cdot)$ for all $i\geq 2$, then $\{X_i\}_{i\in\bn}$ is independent and identically distributed (i.i.d.) under $\mu$. This formulation inspires us to introduce a basic model on canonical space $(\br^\bn,\cb(\br^\bn))$ via probability kernels such that the canonical random variables $\{X_i\}_{i\in\bn}$ are independence and identically distributed with model uncertainty as described in Peng \cite{P2019}, i.e., all the probability kernels belong to a given set of probability measures. Such formulation provides a new interpretation of i.i.d. sequence on sublinear expectation space.

The i.i.d. assumption is usually used in the central limit theorem (CLT). Peng \cite{P2019b} initially introduced the new notion of i.i.d. sequence on the sublinear expectation space $(\Omega,\mathcal{H},\be)$ and  corresponding CLT with zero-mean is established in Peng \cite{P08,pengsur,P09,P2019,P2019b}, known as Peng's CLT:

For an i.i.d. sequence $\{X_i\}_{i\in\bn}$ with $\be[X_1]=\be[-X_1]=0$ and
\begin{equation}\label{uc}
\lim_{\lambda\rightarrow\infty}\mathbb{\hat{E}}[(|X_1|^2-\lambda)^+]=0
\end{equation}
we have
$$\lim_{n\rightarrow\infty}\mathbb{\hat{E}}\lt[\vp\lt(\frac{\sum_{i=1}^nX_i}{\sqrt{n}}\rt)\rt]=\eg[\vp(\xi)], \ \ \forall \vp\in C_{b.Lip}(\br),$$
where $\eg$ is the $G$-expectation corresponding to the $G$-normally distributed random variable $\xi$, which is characterized by the so-called $G$-heat equation. The corresponding proof adopted the partial differential equation (PDE) approach which mainly based on the deep result of the interior regularity for PDE (see Krylov \cite{Kry2}). We emphasize that the regularity of $\be$ (see Definition \ref{reg}) is not required here. {In particular, the CLT holds when $\be$ is only finitely additive without the assumption of regularity. In this case, there does not exist probability measure such that underlying random variables are i.i.d., therefore the classical probability theory is not applicable (see Example \ref{ex1}).}

After Peng established CLT on sublinear expectation space, Zhang \cite{Zhang} obtained the sufficient and necessary conditions of CLT for i.i.d. random variables, in which the condition (\ref{uc}) can be weakened to $\lim_{\lambda\rightarrow\infty}\lambda\mathbb{V}(|X_1|^2\geq\lambda)=0$, where $\mathbb{V}$ is the capacity introduced by $\mathbb{\hat{E}}$. Peng \cite{P09} proposed a new proof of CLT by weakly compact method but involving PDE approach to characterize the $G$-normal distribution. Song \cite{song2} provided an estimate of the convergence rate for CLT by Stein's method, and Krylov \cite{Kry1} gave error estimates in CLT for not necessarily non-degenerate case by the finite-difference approximations for Bellman's equation. Besides, convergence to the $G$-normal distribution also occurs for non-independent random variables or for non-identical distributions (see Li \cite{Li2015}, Li and Shi \cite{LS}, Zhang \cite{zhang2020}). Zhang \cite{Z2015} also considered the functional CLT based on Peng's CLT in \cite{P2019}. We note that all papers on CLT under sublinear expectation are assumed that the underlying random variables are zero-mean and most of proofs rely on PDE approach.

{An interesting problem is that how about CLT for the random variables with mean-uncertainty? Recently, Chen and Epstein \cite{CE} proved a CLT for random variables with mean-uncertainty and unambiguous conditional variance, where the limit is defined by a backward stochastic differential equation (see Pardoux and Peng \cite{PP}, Peng \cite{P97}).  Chen et al. \cite{CEZ} further established a CLT under the assumption that conditional variances can vary in a fixed interval, and the corresponding limit distribution can be calculated by the oscillating Brownian motion.  In \cite{CE} and \cite{CEZ}, one important assumption is that all measures in the set $\cp$  are equivalent on the filtration, while such assumption is not necessary in our paper. Fand et al. \cite{FPSS} also considered the CLT with mean-uncertainty with additional assumptions, which converges to the classical normal distribution.}

Motivated by the weak  approximation of $G$-expectations introduced by Dolinsky et al. \cite{DNS}, we establish the functional CLT on the canonical space $(\br^\bn,\mathcal{B}(\br^\bn), \cp)$, where $\cp$ is a set of probability measures introduced via probability kernels with model uncertainty. The proof relies on the martingale central limit theorem and stability of stochastic integral in the classical probability theory. The CLTs obtained in this paper are in the functional forms, and also take the mean-uncertainty into consideration due to the well-defined upper and lower variances. It is easy to extend functional CLT from the canonical space to the general sublinear expectation space by the representation theorem. Thus we provide a new proof of CLT in \cite{P2019}, which is purely probabilistic, then generalize it to the mean-uncertainty case. We note that the corresponding limit distribution is still $G$-normally distributed. This new CLT illustrates the broad applicability of $G$-normal distribution for the situations with mean-uncertainty. It is interesting that the variant form of such CLT with mean-uncertainty provides a simple proof of Erd\"{o}s-Kac Theorem in number theory (see Guo et al. \cite{gll2}).

This paper also provides a new methodology to study limit theorems on general sublinear expectation space. We firstly establish limit theorems on the canonical space. In this step, the classical martingale limit theorem can be applied to derive the desired limit theorems with model uncertainty, and then we extend them to the general sublinear expectation space by the representation theorem.

The remainder of this paper is organized as follows. Section 2 describes the basic model for i.i.d. sequence on canonical space $(\br^\bn,\cb(\br^\bn))$ introduced by probability kernels with model uncertainty. The corresponding functional CLT with mean-uncertainty are established in Section 3. In Section 4, we extend functional CLT from canonical space to general sublinear expectation space. Some examples are provided in Section 5. The law of large numbers, which is used in the proof of functional CLT, is given in Appendix.

\section{Basic model on canonical space}

Let $(\br^\bn,\cb(\br^\bn))$ be the canonical space and $\{X_i\}_{i\in\bn}$ the sequence of canonical random variables defined by $X_i(\omega)=\omega_i$ for $\omega=(\omega_1,\cdots,\omega_n,\cdots)\in\br^\bn$. For each $i\in\bn$, $\cp_i$ is the convex and weakly compact set of probability measures on $(\br,\mathcal{B}(\br))$ characterizing the uncertainty of the distributions of $X_i$. We define a set of joint laws on $(\br^\bn,\mathcal{B}(\br^\bn))$ via the probability kernels as following:
\begin{equation}
\begin{array}
[c]{l}
\mathcal{P}=\{P\ \text{is\ probability\ measure\ on } (\br^{\bn},\cb(\br^\bn))\ \text{such\ that }
\\P(A^{(n)}\times\br^{\bn-n})=\int_\br\mu_1(dx_1)\int_{\br}\kappa_{2}(x_1,dx_2)\cdots\int_\br I_{A^{(n)}}\kappa_n(x_1,\cdots,x_{n-1},dx_{n})\\
\forall n\geq 1, \ \forall A^{(n)}\in\cb(\br^n),\ \ \kappa_i(x_1,\cdots,x_{i-1},\cdot)\in\cp_i,\ \ 2\leq i\leq n,\ \mu_1\in\cp_1\},
\end{array}
\label{e1}%
\end{equation}
where for each $i\in\bn$, $\kappa_i(x_1,\cdots,x_{i-1}, dx_i)$ is the probability kernel satisfying:
\begin{itemize}
\item[(i)] $\forall (x_1,\cdots,x_{i-1})\in\br^{i-1}$, $\kappa_i(x_1,\cdots,x_{i-1},\cdot)$ is a probability measure on $(\br,\mathcal{B}(\br))$.
\item[(ii)] $\forall B\in\mathcal{B}(\br)$, $\kappa_i(\cdot, B)$ is $\mathcal{B}(\br^{i-1})$-measurable.
\end{itemize}
The existence of such $P$ on $(\br^\bn,\cb(\br^\bn))$ is shown by Ionescu-Tulcea theorem.

{
We note that the set $\cp_i$ is independent of $(x_1,\cdots,x_{i-1})$, i.e., the uncertainty of the distributions of $X_i$ does not change for any realization of history $(x_1,\cdots,x_{i-1})$ of the random vector $(X_1,\cdots,X_{i-1})$, thus we call  $X_i$  independent of $(X_1,\cdots,X_{i-1})$. Here independence means the uncertainty of the distributions of $X_i$ is independent of random vector $(X_1,\cdots, X_{i-1})$. $\{X_i\}_{i\in\bn}$ can be regarded as the canonical stochastic process with discrete time, it is obvious that such independence is not symmetric in general. Furthermore, if the uncertainties of each $X_i$ are the same, i.e., $\cp_i=\cp_1$, $\forall i\geq 2$, we say that $\{X_i\}_{i\in\bn}$ is an i.i.d. sequence. In this case, $\cp$ is determined by the set $\cp_1$ of marginal laws for $X_1$.}

\begin{remark}
In particular, if each $\cp_i=\{\mu_i\}$ is a singleton, above construction is just the classical procedure to construction i.i.d. sequence on the product space by the product measure $\mu=\otimes_{i=1}^\infty\mu_i$.
\end{remark}

For each $\mathcal{B}(\br^\bn)$-measurable random variable $X$, we introduce the sublinear expectation $\pe$ defined by
$$\mathbb{E}^{\cp}[X]:=\sup_{P\in\cp}E_P[X].$$
The canonical filtration $\{\cf_i\}_{i\in\bn}$ is defined as $\cf_i=\sigma(X_k, 1\leq k\leq i)$ with convention $\cf_0=\{\emptyset,\br^\bn\}$. Then we have the following proposition.

\begin{proposition}\label{prop23}
For each $P\in\cp$ and $\vp\in C_{b}(\br)$, we have
\begin{equation}-\mathbb{E}^{\cp}[-\vp(X_i)]\leq E_P[\vp(X_i)|\cf_{i-1}]\leq\mathbb{E}^{\cp}[\vp(X_i)], \ \ P-\text{a.s.}, \ \forall i\in\bn.\label{eq2}\end{equation}
\end{proposition}

\begin{proof}
For each $P\in\cp$, by the regularity of condition probability, almost surely we have,  for $\omega=(x_1,\cdots,x_n,\cdots)\in\br^\bn$,
$$E_P[\vp(X_i)|\mathcal{F}_{i-1}](\omega)=E_{\kappa_i(x_1,\cdots,x_{i-1}, \cdot)}[\vp(X_i)],$$
then (\ref{eq2}) holds since $\kappa_i(x_1,\cdots,x_{i-1}, \cdot)\in\cp_i$.
\end{proof}

We note that (\ref{eq2}) is equivalent to the classical notion of  independence when  $\cp$ is a singleton. In fact, (\ref{eq2}) also holds on the sublinear expectation space, see Guo and Li \cite{GL}.
\begin{remark}
If we further assume that
\begin{equation}
\label{cond11}
\lim_{\lambda\rightarrow\infty}\sup_{i\in\bn}\pe[(|X_i|^2-\lambda)^+]=0,
\end{equation}
then (\ref{eq2}) holds for continuous function $\vp$ satisfying quadratic growth condition $|\vp(x)|\leq C(1+|x|^2)$, where $C$ is a constant. In particular, if $\pe[X_i]=\pe[-X_i]=0$ for all $i\in\bn$, then for each $P\in\cp$,
$$E_P[X_i|\cf_{i-1}]=0,\ \ P-\text{a.s.}$$
\end{remark}
\begin{remark}
In fact, (\ref{eq2}) still holds without the assumptions of convexity and weakly compactness for each $\cp_i$. The sequence $\{X_i\}_{i\in\bn}$ satisfying (\ref{eq2}) is called pseudo-independent, see \cite{GL}.
\end{remark}

\section{Functional central limit theorem on the canonical space}

In this section, we establish the functional CLT on the canonical space, which converges to the $G$-expectation constructed by $G$-heat equation. Here we use another representation of $G$-expectation obtained in Denis et al. \cite{DHP11} and then developed by  Dolinsky et al. \cite{DNS}. We adopt the formulations in \cite{DNS}.
\subsection{Representation of $G$-expectation}
Let $\Omega_0=C([0,1])$ be the space of all continuous paths $\omega=(\omega_t)_{0\leq t\leq 1}$ with $\omega_0=0$, endowed with the norm $||\omega||_\infty=\sup_{0\leq t\leq 1}|\omega_t|$. The canonical process $B$ is defined by $B_t(\omega)=\omega_t$ and its quadratic variation process is denoted by $\langle B\rangle$. The canonical filtration is $\cf^B_t=\sigma(B_s,0\leq s\leq t).$

For the fixed non-negative interval $\Theta$, we set
$$\cp_\Theta=\lt\{P: B\ \text{is}\ \text{martingale\ under\ } P \ \text{and}\ \frac{d\langle B\rangle_t}{dt}\in\Theta, P\times dt-\text{a.e.}  \rt\}$$

Let $P_W$ be the Wiener measure on $C([0,1])$ and we define
$$\mathcal{Q}_\Theta=\lt\{P_W\circ\lt(\int f(t,B)dB_t\rt)^{-1}: f\in\mathcal{A}_\Theta \rt\},$$
where $\mathcal{A}_\Theta$ is the collection of all adapted continuous functions on $[0,1]\times\Omega_0$ taking values in $\sqrt{\Theta}$.

It is proved in \cite{DNS} that upper expectations of $\cp_\Theta$ and $\mathcal{Q}_\Theta$ coincide, i.e.,
$$\mathbb{E}^{\cp_\Theta}[\vp]=\sup_{P\in\cp_\Theta}E_P[\vp]=\sup_{P\in\mathcal{Q}_\Theta}E_P[\vp]=\mathbb{E}^{\mathcal{Q}_\Theta}[\vp], \ \ \ \forall \vp\in C_b(\Omega_0).$$

\subsection{Upper and lower variances with mean-uncertainty}

Let $\cp_0$ be a weakly compact and convex set of probability measures on $(\br,\cb(\br))$ and $X$ be the canonical random variable, i.e., $X(\omega)=\omega, \ \forall \omega\in\br$.
We define
$$\mathbb{E}^{\cp_0}[\vp(X)]=\sup_{P\in\cp_0}E_P[\vp], \ \ \forall \vp\in C(\br).$$

In order to deal with mean-uncertainty case, i.e., $\mathbb{E}^{\cp_0}[X]>-\mathbb{E}^{\cp_0}[-X]$, we need to define the proper upper and lower variances.

We denote $V_P(X)$ the classical variance of $X$ under probability measure $P$, then
$$V_P(X)=E_P[(X-E_P[X])^2]=\min_{\mu\in\br}E_P[(X-\mu)^2].$$
It is natural to introduce the upper and lower variances under sublinear expectation $\mathbb{E}^{\cp_0}$ instead of $E_P$ (see Walley \cite{walley} or Li et al. \cite{lly}).

\begin{definition}\label{uvc}
For the canonical random variable $X$ on $(\br,\cb(\br))$ with $\mathbb{E}^{\cp_0}[|X|^2]<\infty$, define the upper variance of $X$ to be
$$\overline{V}(X):=\min_{\lu\leq\mu\leq\ou}\{\mathbb{E}^{\cp_0}[(X-\mu)^2]\},$$
and the lower variance of $X$ to be
$$\underline{V}(X):=\min_{\lu\leq\mu\leq\ou}\{-\mathbb{{E}}^{\cp_0}[-(X-\mu)^2]\},$$
where $\ou=\mathbb{E}^{\cp_0}[X]$ and $\lu=-\mathbb{E}^{\cp_0}[-X]$.
\end{definition}

The following relation between classical variance and upper and lower variances was proved in \cite{walley} (see also \cite{lly}).
\begin{proposition}\label{prop34}
\begin{equation}
\overline{V}(X)=\max_{P\in\cp_0}V_P(X). \label{ucc1}
\end{equation}
\begin{equation}
\underline{V}(X)=\min_{P\in\cp_0}V_P(X). \label{ucc2}
\end{equation}
\end{proposition}
\begin{remark}
The weak compactness and convexity of $\cp_0$ ensure that the minimax theorem in Sion \cite{sion} can be applied to prove (\ref{ucc1}). For the general $\cp_0$ without such assumptions, (\ref{ucc2}) still holds but (\ref{ucc1}) becomes the following inequality:
$$\overline{V}(X)\geq\sup_{P\in\cp_0}V_P(X).$$
\end{remark}

Let $\cp$ be constructed on $(\br^\bn,\cb(\br^\bn))$ by (\ref{e1}) with $\cp_i=\cp_0$ for $i\in\bn$, where $\cp_0$ is weakly compact and convex. We further assume that $\cp_0$ satisfies
\begin{equation}
\lim_{\lambda\rightarrow\infty}\mathbb{E}^{\cp_0}[(|X|^{2}-\lambda)^+]=0,\label{c2}
\end{equation}
where $X(\omega)=\omega, \ \forall \omega\in\br$.

Let $\{X_i\}_{i\in\bn}$ be the  canonical process on the canonical space $(\br^\bn,\cb(\br^\bn))$ with $X_i(\omega)=\omega_i, \ \forall \omega=(\omega_1,\cdots,\omega_n,\cdots)\in\br^\bn$. Then $\{X_i\}_{i\in\bn}$ is an i.i.d. sequence under $\cp$. The natural filtration $\{\cf_i\}_{i\in\bn}$ is defined by $\cf_i=\sigma(X_1,\cdots,X_i)$ with convention $\cf_0=\{\emptyset,\br^\bn\}$.

The following proposition is important to prove the functional CLT.

\begin{proposition}\label{lem1}
For each $P\in\cp$, let $\tilde{X}^P_i=X_i-E_P[X_i|\cf_{i-1}]$, then
\begin{equation}\label{conditional1}
\underline{V}(X_1)\leq E_P[|\tilde{X}^P_i|^2|\cf_{i-1}]\leq\overline{V}(X_1),\ \   P-\text{a.s.},  \  \ i\in \bn,
\end{equation}
\end{proposition}
\begin{proof}
If $i=1$, then by Proposition \ref{prop34}, we have
$$\underline{V}(X_1)\leq E_P[(X_1-E_P[X_1])^2]\leq\overline{V}(X_1), \ \ \ \forall P\in\cp.$$

For fixed $i\geq 2$, then $\forall P\in\cp$ and $\omega=(x_1,\cdots,x_{i-1},\cdots)\in\br^\bn$,
\begin{align*}
E_P[|\tilde{X}^P_i|^2|\cf_{i-1}](\omega)=V_{\kappa_i(x_1,\cdots,{x}_{i-1},\ \cdot)}(X_i),\ P-\text{a.s.}.
\end{align*}

Since $\kappa_i(x_1,\cdots,x_{i-1},\cdot)\in\cp_0$,  thanks to Proposition \ref{prop34},  (\ref{conditional1}) holds.

\end{proof}

\subsection{Functional central limit theorem with mean-uncertainty}
In order to obtain the functional  CLT, we need to extend each discrete path $x\in\br^n$ to a continuous path $\hat{x}\in\Omega_0$, where the interpolation operator $\hat{\ }: \br^{n}\rightarrow C([0,1])$ is defined as
$$x=(x_1, \cdots, x_n)\mapsto \hat{x}=(\hat{x}_t)_{0\leq t\leq 1},$$
where $\hat{x}_t:=([nt]+1-nt)x_{[nt]}+(nt-[nt])x_{[nt]+1}$ with $x_0=0$.

With above notations, we have the functional CLT with mean-uncertainty on the canonical space $(\br^\bn,\cb(\br^\bn))$.

\begin{theorem}\label{fclt}
For each $P\in\cp$, let $S_P^{(n)}=(S^P_i)_{1\leq i\leq n}$ with $S^P_i=\frac{1}{\sqrt{n}}\sum_{j=1}^i(X_j-E_P[X_j|\cf_{j-1}])$. Then for each continuous function $\vp:C([0,1])\rightarrow\br$ satisfying $|\vp(\omega)|\leq C(1+||\omega||_\infty)$ for some constant $C>0$, we have
\begin{equation}\label{clt0}
\lim_{n\rightarrow\infty}\sup_{P\in\cp}E_P\lt[\vp\lt(\hat{S_P^{(n)}}\rt)\rt]=\mathbb{E}^{\cp_\Theta}[\vp],
\end{equation}
where $\Theta=[\underline{V}(X_1),\overline{V}(X_1)]$.
\end{theorem}
\begin{proof}
We firstly prove that (\ref{clt0}) holds for $\vp\in C_{b.Lip}(\Omega_0)$, where $C_{b.Lip}(\Omega_0)$ is the space of all bounded and Lipschitz continuous functions on $\Omega_0$.

Let us consider the first inequality
\begin{equation}
\limsup_{n\rightarrow\infty}\sup_{P\in\cp}E_P\lt[\vp\lt(\hat{S_P^{(n)}}\rt)\rt]\leq\mathbb{E}^{\cp_\Theta}[\vp].\label{iq1}
\end{equation}
For fixed $\vp\in C_{b.Lip}(\Omega_0)$ and $\ve>0$, there exists $P^{(n)}\in\cp$ such that, for each $n\in\bn$,
$$E_{P^{(n)}}\lt[\vp\lt(\hat{S_{P^{(n)}}^{(n)}}\rt)\rt]\geq\sup_{P\in\cp}E_P\lt[\vp\lt(\hat{S_P^{(n)}}\rt)\rt]-\ve.$$

Let $Q^{(n)}$ be the law of $\hat{S_{P^{(n)}}^{(n)}}$ introduced by probability measure $P^{(n)}\in\cp$, i.e.,
$$E_{Q^{(n)}}[\vp(B)]=E_{P^{(n)}}\lt[\vp\lt(\hat{S_{P^{(n)}}^{(n)}}\rt)\rt], \ \ \forall \vp\in C_b(\Omega_0).$$
The sequence $\{Q^{(n)}\}$ is tight on $C[0,1]$ by condition (\ref{c2}) (see Zhang \cite{Z2015}). Let $Q$ be a cluster point of $Q^{(n)}$, then the canonical process $B$ is a $Q$-martingale.

Let $M^{(n)}(t)=\frac{1}{\sqrt{n}}\sum_{i=0}^{[nt]}\tilde{X}_i^{P^{(n)}}$, where $\tilde{X}_i^{P^{(n)}}=X_i-E_{P^{(n)}}[X_i|\cf_{i-1}]$ and $X_0=0$. The law of $M^{(n)}$ on the space $D[0,1]$ of c\`{a}dl\`{a}g paths is denoted by $\tilde{Q}^{(n)}$. We note that
$$\lt|E_{P^{(n)}}\lt[\vp\lt(\hat{S_{P^{(n)}}^{(n)}}\rt)\rt]-E_{P^{(n)}}[\vp(M^{(n)})]\rt|\leq\frac{L_\vp}{\sqrt{n}}E_{P^{(n)}}[|\tX^{P^{(n)}}_i|]\leq\frac{2L_\vp}{\sqrt{n}}\mathbb{E}^{\cp_0}[|X_1|]\rightarrow 0,$$
where $L_\vp$ is the Lipschitz constant of $\vp$. Thus $Q$ is also the cluster point of $\tilde{Q}^{(n)}$.

We have
$$\langle M^{(n)}\rangle_t=\frac{1}{n}\sum_{i=0}^{[nt]}|\tilde{X}^P_i|^2,$$
and define $d_\Theta(x)=\inf_{y\in\Theta}|y-x|$.

Since $\underline{V}(X_1)\leq E_{Q^{(n)}}[|\tilde{X}^P_i|^2|\mathcal{F}_{i-1}]\leq\overline{V}(X_1)$ by Proposition \ref{lem1}, the similar argument of the law of large numbers  (see Appendix) shows that, for any $0\leq s<t\leq 1$,

$$E_{\tilde{Q}^{(n)}}\lt[d_\Theta\lt(\frac{\langle M^{(n)}\rangle_t-\langle M^{(n)}\rangle_s}{t-s}\rt)\rt]\rightarrow 0,$$
which implies that
$$E_{Q}\lt[d_\Theta\lt(\frac{\langle B\rangle_t-\langle B\rangle_s}{t-s}\rt)\rt]=0,$$
thus $Q\in \cp_\Theta$ and (\ref{iq1}) holds.

Now we consider the second inequality
\begin{equation}\label{iq2}
\liminf_{n\rightarrow\infty}\sup_{P\in\cp}E_P[\vp(\hat{S_P^{(n)}})]\geq\mathbb{E}^{\cp_\Theta}[\vp].
\end{equation}
Let $\vp\in C_{b.Lip}(\Omega_0)$ be fixed. For each $\ve>0$, there exists $\bar{f}\in \mathcal{A}_{\Theta}$, such that
$$E_{P_W}\lt[\vp\lt(\int\bar{f}(t,W)dW_t\rt)\rt]\geq \mathbb{E}^{\mathcal{Q}_\Theta}[\vp]-\ve.$$
Without loss of generality, we further assume that $\bar{f}>0$.

By Proposition \ref{prop34}, we can find probability measures $\overline{P}$ and $\underline{P}$ on $(\br,\mathcal{B}(\br))$ such that $V_{\overline{P}}(X_1)=\overline{V}(X_1)$ and $V_{\underline{P}}(X^n_1)=\underline{V}(X_1)$ respectively.

For fixed $n\in\bn$, since $\bar{f}(0,\cdot)\in\sqrt{\Theta}$ is a constant,
 we define, $\forall B\in\mathcal{B}(\br)$, if $\overline{V}(X_1)>\underline{V}(X_1)$,
$$\mu_1(B)=\frac{\overline{V}(X_1)-\bar{f}(0,\cdot)^2}{\overline{V}(X_1)-\underline{V}(X_1)}\underline{P}(B)+\frac{\bar{f}(0,\cdot)^2-\underline{V}(X_1)}{\overline{V}(X_i)-\underline{V}(X_1)}\overline{P}(B),$$
otherwise, $\mu_1(B)=\overline{P}(B)$.

There exists $P_*\in\cp$ such that
\begin{equation}\label{ce}
{E_{P_*}[|X_1-E_{P_*}[X_1]|^2]}=\bar{f}\lt(0, \cdot\rt)^2\subset \Theta.
\end{equation}
We then define
$$Y_1=\frac{X_1-E_{P_*}[X_1]}{\sqrt{E_{P_*}[|X_1-E_{P_*}[X_1]|^2]}}.$$

Let $W^{(n)}$ be defined recursively by
$$W^{(n)}_t=\frac{1}{\sqrt{n}}\sum_{j=0}^{[nt]}Y_j,$$
where $Y_0=0$, and for $j\geq 1$,
$$Y_j=\frac{X_j-E_{P_*}[X_j|\cf_{j-1}]}{\sqrt{E_{P_*}[|X_j-E_{P_*}[X_j|\cf_{j-1}]|^2|\cf_{j-1}]}}.$$

Indeed, for $j\geq 2$, we note that $\bar{f}(\frac{j-1}{n},\hat{W^{(n)}})$  depends on $\{Y_i\}_{0\leq i\leq j-1}$ by the adaptability of $\bar{f}$. We define, $\forall B\in\cb(\br)$, if $\overline{V}(X_1)>\underline{V}(X_1)$,
\begin{align*}\kappa_j(x_1,\cdots,x_{j-1},B)=&\frac{\overline{V}(X_1)-\bar{f}(\frac{j-1}{n},\hat{W^{(n)}})^2}{\overline{V}(X_1)-\underline{V}(X_1)}\underline{P}(B)\\
&+\frac{\bar{f}(\frac{j-1}{n},\hat{W^{(n)}})^2-\underline{V}(X_1)}{\overline{V}(X_1)-\underline{V}(X_1)}\overline{P}(B),
\end{align*}
otherwise, $\kappa_j(x_1,\cdots,x_{j-1},B)=\overline{P}(B)$.

$P_*$ can be formulated by, $\forall A\in\br^n$,
$$P_*(A\times\br^{\bn-n})=\int_\br\mu_1(dx_1)\int_\br\kappa_2(x_1,dx_2) \cdots\int_\br I_A\kappa_n(x_1,\cdots,x_{n-1},dx_n).$$

We can verify that for $1\leq i\leq n$,
\begin{equation*}
{E_{P_*}[|X_i-E_{P_*}[X_i|\cf_{i-1}]|^2|\mathcal{F}_{i-1}]}=\bar{f}\lt((i-1)/n, \hat{W^{(n)}}\rt)^2\subset \Theta.
\end{equation*}

It is clear that $E_{P_*}[Y_j|\cf_{j-1}]=0$ and $E_{P_*}[Y_j^2|\cf_{j-1}]=1$, $j\geq 1$.
By the martingale central limit theorem in Brown \cite{brown}, on the space of c\`{a}dl\`{a}g paths $D([0,1],\br^2)$ equipped with the Skorohod topology,
$$\lt(W^{(n)},\hat{W^{(n)}}\rt)\Rightarrow(W,W),$$
where $W$ is the Brownian motion.

We note that
\begin{align*}
S^{P_*}_i&=\frac{1}{\sqrt{n}}\sum_{j=1}^i(X_j-E_{P_*}[X_j|\cf_{j-1}])\\
&=\sum_{j=1}^i\bar{f}\lt(\frac{j-1}{n}, \hat{W^{(n)}}\rt)\lt(W_{\frac{j+1}{n}}^{(n)}-W_{\frac{j}{n}}^{(n)}\rt).
\end{align*}
Then it follows Dolinsky et al. \cite{DNS} and the stability of stochastic integrals (see Kurtz and Protter \cite{KP}), on $D(0,1)$,
$$\hat{S_P^{(n)}}\Rightarrow \int\bar{f}(t,W)dW_t.$$
 Finally, we obtain
\begin{align*}
\liminf_{n\rightarrow\infty}\sup_{P\in\cp}E_P\lt[\vp\lt(\hat{S_P^{(n)}}\rt)\rt]&\geq\liminf_{n\rightarrow\infty}E_{P_*}\lt[\vp\lt(\hat{S_P^{(n)}}\rt)\rt]\\
&\geq E_{P_W}\lt[\vp\lt(\int_0^1\bar{f}(t,W)dW_t\rt)\rt]\geq \mathbb{E}^{\mathcal{Q}_\Theta}[\vp]-\ve.
\end{align*}
Thus (\ref{clt0}) holds for $\vp\in C_{b.Lip}(\Omega_0)$.

For fixed $n\in\bn$, by Doob's martingale inequality, we have
$$E_{P}[\sup_{0\leq t\leq 1}|M^{(n)}(t)|^2]\leq 4\overline{V}(X_1), \ \ \ \forall P\in\cp,$$
which implies that
$$\sup_{P\in\cp}E_P\lt[\sup_{0\leq t\leq 1}\lt|\hat{S_P^{(n)}}(t)\rt|^2\rt]\leq 4\overline{V}(X_1)+2\mathbb{E}^{\cp}[|X_1|^2].$$
By the similar arguments as Lemma 2.4.12 in Peng \cite{P2019}, (\ref{clt0}) also holds for continuous function $\vp$ with linear growth condition.
\end{proof}

\begin{remark}
The first inequality (\ref{iq1}) can also be proved by the similar argument in \cite{DNS}. For the second inequality (\ref{iq2}), the corresponding proof in \cite{DNS} can not be applied directly, since the discrete-time models in \cite{DNS} do not have independence property.
\end{remark}

\section{Functional central limit theorem on sublinear expectation space}

In this section, we extend  functional CLT to the general sublinear expectation space $(\Omega,\mathcal{H},\be)$ introduced by Peng \cite{P2019}.
\subsection{Basic notions of sublinear expectation theory}
Let $\Omega$ be a given set and let $\mathcal{H}$ be a linear space of real
functions defined on $\Omega$ such that if
$X_{1},\ldots,X_{n}\in \mathcal{H}$, then $\varphi(X_{1},\cdots,X_{n}%
)\in \mathcal{H}$ for each $\varphi \in C_{Lip}(\mathbb{R}^{n})$, where
$C_{Lip}(\mathbb{R}^{n})$ denotes the space of all Lipschitz
functions on $\mathbb{R}^{n}$. $\mathcal{H}$ is considered as the space of
random variables.
$X=(X_{1},\ldots,X_{n})$, $X_{i}\in \mathcal{H}$, is called
an $n$-dimensional random vector.

\begin{definition}\label{sub}
A sublinear expectation $\hat{\mathbb{E}}$ on $\mathcal{H}$ is a functional $\hat{\mathbb{{E}}}:\mathcal{H}\rightarrow \mathbb{R}$ satisfying the following
properties: for all $X,Y\in \mathcal{H}$, we have
\begin{description}
\item[(a)] Monotonicity: $\mathbb{\hat{E}}[X]\geq \mathbb{\hat{E}}[Y]$ if
$X\geq Y$.

\item[(b)] Constant preserving: $\mathbb{\hat{E}}[c]=c$ for $c\in \mathbb{R}$.

\item[(c)] Sub-additivity: $\mathbb{\hat{E}}[X+Y]\leq \mathbb{\hat{E}%
}[X]+\mathbb{\hat{E}}[Y]$.

\item[(d)] Positive homogeneity: $\mathbb{\hat{E}}[\lambda X]=\lambda
\mathbb{\hat{E}}[X]$ for $\lambda \geq0$.
\end{description}

The triple $(\Omega,\mathcal{H},\mathbb{\hat{E}})$ is called a sublinear
expectation space.
\end{definition}


Let $X=(X_1,\cdots,X_n)$ be a given $n$-dimensional random vector on a sublinear expectation space $(\Omega,\mathcal{H},\mathbb{\hat{E}})$. We define a functional on $C_{b.Lip}(\br^n)$ by
$$\mathbb{\hat{F}}_X[\vp]:=\mathbb{\hat{E}}[\vp(X)], \ \ \forall \vp\in C_{b.Lip}(\br^n).$$
The triple $(\br^n, C_{b.Lip}(\br^n),\mathbb{\hat{F}}_X[\cdot])$ forms a sublinear expectation space, and $\mathbb{\hat{F}}_X$ is called the sublinear distribution of $X$.

\begin{definition}\label{id}
Let $X$ and $Y$ be two random variables on $(\Omega,\mathcal{H},\mathbb{\hat
{E}})$. $X$ and $Y$ are called identically distributed, denoted by
$X\overset{d}{=}Y$, if for each $\varphi \in C_{b.Lip}(\mathbb{R})$,
\[
\mathbb{\hat{E}}[\varphi(X)]=\mathbb{\hat{E}}[\varphi(Y)].
\]

\end{definition}

\begin{definition}
\label{new-de1}Let $\left \{  X_{n}\right \}  _{n\in\bn}$ be a sequence of
random variables on $(\Omega,\mathcal{H},\mathbb{\hat{E}})$. $X_{n}$ is said
to be independent of $\left(  X_{1},\ldots,X_{n-1}\right)  $ under
$\mathbb{\hat{E}}$, if for each $\vp\in C_{b.Lip}(\br^{n})$,
$$
\mathbb{\hat{E}}\left[\varphi\left(X_{1}, \cdots, X_{n}\right)\right]=\mathbb{\hat{E}}\left[\left.\mathbb{\hat{E}}\left[\varphi\left(x_{1}, \cdots, x_{n-1},X_n\right)\right]\right|_{\left(x_{1}, \cdots, x_{n-1}\right)=\left(X_{1}, \cdots, X_{n-1}\right)}\right].
$$
The sequence of random variables $\left \{  X_{n}\right \}  _{n\in\bn}$ is
said to be independent, if $X_{n+1}$ is independent of $\left(  X_{1}%
,\ldots,X_{n}\right)  $ for each $n\geq1$.
\end{definition}

\begin{remark}
This definition of independence is more general than the classical one, since we do not require the regularity of $\be$ (see Definition \ref{reg}). Example \ref{ex1} illustrates that we can not construct i.i.d. sequence on the classical probability space, but we can do it on the sublinear expectation space.
\end{remark}

The following proposition gives the links between the i.i.d. sequence on sublinear expectation space and that on canonical space.
\begin{proposition}\label{pp1}
Let $\{X_i\}_{i\in\bn}$ be an i.i.d. sequence on sublinear expectation space $(\Omega,\mathcal{H},\be)$. Then there exists a sequence of i.i.d. canonical random variables  $\{Y_i\}_{i\in\bn}$ on canonical space $(\br^\bn,\cb(\br^\bn))$ associated with the set of probability measures $\cp$ such that for each $n\in\bn$  and $\vp\in C_{b.Lip}(\br^n)$,
$$\be[\vp(X_1,\cdots,X_n)]=\mathbb{E}^{\cp}[\vp(Y_1,\cdots,Y_n)].$$
\end{proposition}
\begin{proof}
In fact, we only need to prove that
\begin{equation}\label{iddd}
\be[\vp(X_1)]=\pe[\vp(Y_1)], \ \ \forall \vp\in C_{b.Lip}(\br),
\end{equation}
and $\{Y_i\}_{i\in\bn}$ is an i.i.d. sequence defined by Definition \ref{id} and \ref{new-de1}.

To prove (\ref{iddd}), let $\cp_0$ be the set of probability measures on $(\br,\cb(\br))$ defined by
$$\cp_0=\{P: E_P[\vp]\leq\be[\vp(X_1)], \ \vp\in C_{b.Lip}(\br)\}.$$
It is clear that  $\cp_0$ is weakly compact and convex, and
$$\be[\vp(X_1)]=\max_{P\in\cp_0}E_P[\vp], \ \ \forall \vp\in C_{b.Lip}(\br).$$

Let $\cp$ be the set of probability measures on canonical space $(\br^\bn,\cb(\br^\bn))$ constructed by (\ref{e1}) with $\cp_i=\cp_0$, $i\in\bn$. Obviously,
$$\be[\vp(X_1)]=\pe[\vp(Y_1)].$$

To prove the independence of $\{Y_i\}_{i\in\bn}$, for simplicity, we only prove the case of $n=2$. In this case, we do not distinguish canonical random variables $Y_1$ and $Y_2$ on the canonical spaces $(\br^\bn,\cb(\br^\bn))$ or on $(\br^2,\cb(\br^2))$.

By the representation theorem of sublinear expectation (cf. \cite{HL,HLL,HP09}), there exists $\tilde{\cp}$ on $(\br^2,\cb(\br^2))$ defined by
\begin{equation}\label{rep}
\tilde{\cp}=\{P: E_P[\vp]\leq\be[\vp(X_1,X_2)], \ \forall \vp\in C_{b.Lip}(\br^2)\},
\end{equation}
such that
$$\be[\vp(X_1,X_2)]=\max_{P\in\tilde{\cp}}E_P[\vp],\ \ \forall \vp\in C_{b.Lip}(\br^2).$$

Let $\cp^{(2)}$ be the set on $(\br^2,\cb(\br^2))$ constructed by (\ref{e1}) with $\cp_1=\cp_2=\cp_0$. We claim that $\cp^{(2)}=\tilde{\cp}$.

For each $P\in\cp^{(2)}$, we have
\begin{align*}
E_P[\vp(Y_1,Y_2)]&=E_P[E_P[\vp(Y_1,Y_2)|Y_1]]\\
&=E_P[E_P[\vp(y_1,Y_2)|Y_1]|_{y_1=Y_1}]\\
&\leq E_P[\mathbb{E}^{\cp_0}[\vp(y_1,Y_2)]|_{y_1=Y_1}]\\
&\leq\mathbb{E}^{\cp_0}[\mathbb{E}^{\cp_0}[\vp(y_1,Y_2)]|_{y_1=Y_1}]=\be[\vp(Y_1,Y_2)],
\end{align*}
which implies $\cp^{(2)}\subset\tilde{\cp}$.

On the other hand, for fixed $P\in\tilde{\cp}$, $\forall A\in\cb(\br^2)$, we have the following decomposition:
$$P(A)=\int_\br\int_\br I_A\kappa_2(x_1,dx_2)\mu_1(dx_1).$$
It is clear that
$$E_{\mu_1}[\vp]=E_P[\vp]\leq\mathbb{E}^{\cp_0}[\vp], \ \forall \vp\in C_{b.Lip}(\br),$$
then by Corollary 2.8 in Li and Lin \cite{LL}, we obtain $\mu_1\in\cp_0$.

For fixed $x_1\in\br$,
$$E_{\kappa_2(x_1,\cdot)}[\vp(x_1,\cdot)]\leq\mathbb{E}^{\cp_0}[\vp(x_1,\cdot)], \ \forall \vp\in C_{b.Lip}(\br^2),$$
thus $\kappa_2(x_1,\cdot)\in\cp_0$. Therefore,
$$\cp^{(2)}=\tilde{\cp},$$
which implies that $Y_2$ is independent of $Y_1$.
\end{proof}

For random variable $X$ on sublinear expectation $(\Omega,\mathcal{H},\be)$, the corresponding upper and lower variances of $X$ can be defined similarly as in Definition \ref{uvc}, i.e.,
$$\overline{V}(X)=\min_{\lu\leq\mu\leq\ou}\be[(X-\mu)^2] \ \ \text{and} \ \  \underline{V}(X)=\min_{\lu\leq\mu\leq\ou}-\be[-(X-\mu)^2].$$

\subsection{Functional central limit theorem under regular sublinear expectation}
We firstly assume that sublinear expectation $\be$ is regular and present the corresponding functional CLT.

\begin{definition}\label{reg}
The sublinear expectation $\be$ is said to be regular if for each sequence $\{X_i\}_{i\in\bn}\subset\mathcal{H}$ with $X_i\downarrow 0$, we have $\be[X_i]\downarrow 0$.
\end{definition}

If $\be$ is regular, we consider a family of probability measures $\cp_*$ defined on the measurable space $(\Omega,\sigma(\mathcal{H}))$:
\begin{equation}\label{eqp}
\cp_*=\{P \ \text{is\ probability\ measure\ on\ }(\Omega,\sigma(\mathcal{H})): E_P[X]\leq \be[X], \ \ \forall X\in\mathcal{H}\},
\end{equation}
where $\sigma(\mathcal{H})$ is the smallest $\sigma$-algebra generated by $\mathcal{H}$.

By the Robust Daniell-Stone theorem in Peng \cite{P2019}, $\cp_*$ is non-empty and
$$\be[X]=\sup_{P\in\cp_*}E_P[X], \ \ \forall X\in\mathcal{H}.$$
Then the domain of $\be$ can be enlarged from $\mathcal{H}$ to $\mathcal{L}(\sigma(\mathcal{H}))$, i.e.,
$$\be[X]=\sup_{P\in\cp_*}E_P[X], \ \forall X\in\mathcal{L}(\sigma(\mathcal{H})),$$
where $\mathcal{L}(\sigma(\mathcal{H}))$ is the space of all $\sigma(\mathcal{H})$-measurable functions.

Proposition \ref{lem1} can be naturally generalized on the sublinear expectation space $(\Omega,\mathcal{H},\be)$ where $\be$ is regular (see Guo et al. \cite{gll}).
\begin{proposition}\label{lem2}
Let $\{X_i\}_{i\in\bn}$ be an i.i.d. sequence on sublinear expectation space $(\Omega,\mathcal{H},\be)$. We further assume that
\begin{equation*}
\lim_{\lambda\rightarrow\infty}\be[(|X_1|^{2}-\lambda)^+]=0.
\end{equation*}
For fixed $n\in\bn$, and for each $P\in\cp_*$, let $\tX^P_i=X_i-E_P[X_i|\cf_{i-1}], \ 1\leq i\leq n$. Then we have
\begin{equation*}
\underline{V}(X_1)\leq E_P[|\tX^P_i|^2|\cf_{i-1}]\leq\overline{V}(X_1),\ \   P-\text{a.s.},  \  \ 1\leq i\leq n.
\end{equation*}
\end{proposition}

The functional CLT with mean-uncertainty on the canonical space can be easily extended to the sublinear expectation space associated with regular sublinear expectation.

\begin{theorem}
Let $\{X_i\}_{i\in\bn}$ be an i.i.d. sequence on sublinear expectation space $(\Omega,\mathcal{H},\be)$ with
$$\lim_{\lambda\rightarrow\infty}\be[(|X_1|^2-\lambda)^+]=0.$$
We further assume that $\be$ is regular and is associated with $\cp_*$ defined by (\ref{eqp}).
For each $P\in\cp_*$, let $\tilde{X}^P_i=X_i-E_P[X_i|\cf_{i-1}]$ and $S^{(n)}_P=(S^P_i)_{1\leq i\leq n}$ with $S^P_i=\frac{1}{\sqrt{n}}\sum_{j=1}^i\tilde{X}^P_j$. Then for each continuous function $\vp:C([0,1])\rightarrow\br$ satisfying $|\vp(\omega)|\leq C(1+||\omega||_\infty)$ for some constants $C>0$,
$$\lim_{n\rightarrow\infty}\sup_{P\in\cp_*}E_P\lt[\vp\lt(\hat{S_P^{(n)}}\rt)\rt]=\mathbb{E}^{\cp_\Theta}[\vp],$$
where $\Theta=[\underline{V}(X_1),\overline{V}(X_1)]$.
\end{theorem}

\begin{proof} Let $\cp_0$ be the weakly compact and convex set of probability measures on $(\br,\cb(\br))$ such that
$$\mathbb{E}^{\cp_0}[\vp]=\be[\vp(X_1)], \ \ \forall \vp\in C_{b.Lip}(\br).$$

For fixed $n\in\bn$,  let $\{Y_i\}_{i\in\bn}$ be the canonical random variables on $(\br^\bn,\cb(\br^\bn))$ and $\cp$ is constructed by (\ref{e1}) with $\cp_i=\cp_0$, $i\in\bn$. Then
$$\be[\vp(X_1,\cdots,X_n)]=\mathbb{E}^{\cp}[\vp(Y_1,\cdots,Y_n)], \ \ \forall \vp\in C_{b.Lip}(\br^n).$$

Furthermore, for each $P\in\cp$, let $\tilde{Y}^P_i=Y_i-E_P[Y_i|\mathcal{G}_{i-1}]$, where $\mathcal{G}_{i-1}=\sigma(Y_i, 1\leq j\leq i)$, and $T_P^{(n)}=(T^P_i)_{1\leq i\leq n}$, where $T^P_i=\frac{1}{\sqrt{n}}\sum_{j=1}^i \tilde{Y}^P_j$.

Similar to Proposition \ref{pp1}, we have
$$\sup_{P\in\cp_*}E_P\lt[\vp\lt(\hat{S_P^{(n)}}\rt)\rt]=\sup_{P\in\cp}E_P\lt[\vp\lt(\hat{T_P^{(n)}}\rt)\rt], \ \forall \vp\in C_{b.Lip}(\Omega_0).$$
Thus Theorem \ref{fcltsub} can be implied by Theorem \ref{fclt}.

\end{proof}

\subsection{Functional central limit theorem without assumption of regularity}
Now we relax the assumption of regularity of $\be$. For fixed $n\in\bn$, let $\cp_*^{(n)}$ denote the class of all probability measures on $(\Omega,\cf_n)$ that dominated by $\be$ in the following sense,
\begin{align*}
\cp_*^{(n)}=\{P: &E_P[\vp(X_1,\cdots,X_n)]\leq\be[\vp(X_1,\cdots,X_n)], \ \forall \vp\in C_{b.Lip}(\br^n)\}.
\end{align*}
where $\cf_n=\sigma(X_1,\cdots,X_n)$ is the natural filtration with convention $\cf_0=\{\emptyset,\Omega\}$.

Then $\be$ is regular on $\mathcal{H}\bigcap\mathcal{L}(\cf_n)$ (see Theorem 10 in Hu and Li \cite{HL}) and the domain of $\be$ can be enlarged naturally from $\mathcal{H}\bigcap\mathcal{L}(\cf_n)$ to $\mathcal{L}(\cf_n)$, which is defined as
$$\be[X]=\sup_{P\in\cp_*^{(n)}}E_P[X], \ \forall X\in\mathcal{L}(\cf_n),$$
where $\mathcal{L}(\cf_n)$ is the space of all $\cf_n$-measurable functions.

We immediately have the functional CLT without the regularity of $\be$ by the similar argument as in the previous subsection.
\begin{theorem}\label{fcltsub}
Let $\{X_i\}_{i\in\bn}$ be an i.i.d. sequence on sublinear expectation space $(\Omega,\mathcal{H},\be)$ with
$$\lim_{\lambda\rightarrow\infty}\be[(|X_1|^2-\lambda)^+]=0.$$
For each $P\in\cp_*^{(n)}$, where $\cp_*^{(n)}$ is dominated by $\be$ on $\mathcal{H}\cap\mathcal{L}(\cf_n)$, let $\tilde{X}^P_i=X_i-E_P[X_i|\cf_{i-1}]$ and $S^{(n)}_P=(S^P_i)_{1\leq i\leq n}$ with $S^P_i=\frac{1}{\sqrt{n}}\sum_{j=1}^i\tilde{X}^P_j$. Then for each continuous function $\vp:C([0,1])\rightarrow\br$ satisfying $|\vp(\omega)|\leq C(1+||\omega||_\infty)$ for some constants $C>0$,
$$\lim_{n\rightarrow\infty}\sup_{P\in\cp_*^{(n)}}E_P\lt[\vp\lt(\hat{S_P^{(n)}}\rt)\rt]=\mathbb{E}^{\cp_\Theta}[\vp],$$
where $\Theta=[\underline{V}(X_1),\overline{V}(X_1)]$.
\end{theorem}

We can generalize CLT in Peng \cite{P2019} from mean-zero case to the case of mean-uncertainty.
\begin{corollary}\label{cltv}
With the same assumptions in Theorem \ref{fcltsub}, then $\forall \vp\in C(\br)$ with linear growth,
\begin{equation}
\label{th1}\lim_{n\rightarrow\infty}\sup_{P\in\cp_*^{(n)}}E_P\lt[\vp\lt(\frac{\sum_{i=1}^n(X_i-E_P[X_i|\cf_{i-1}])}{\sqrt{n}}\rt)\rt]=\eg[\vp(\xi)],
\end{equation}
where $\eg$ is the $G$-expectation corresponding to the $G$-normally distributed random variable $\xi\sim\mathcal{N}(0,[\underline{V}(X_1),\overline{V}(X_1)])$.
\end{corollary}
\begin{remark}
Let $\xi\sim\mathcal{N}(0,[\ls^2,\os^2])$, we have
$$\eg[\vp(\xi)]=\mathbb{E}^{\cp_\Theta}[\vp],$$
where $\Theta=[\ls^2,\os^2]$ and $\vp$ in the right hand can be regarded as $\vp\in C(\Omega_0)$ with $\vp(\omega)=\vp(\omega_1)$,  $\forall \omega\in\Omega_0$.
\end{remark}

{{It is well-known that the conditional expectation is the least-squares-best predictor of $X_i$ based on the realization of the historical data $\{X_j\}_{1\leq j\leq i-1}$, formula (\ref{th1}) means that the properly normalized sum for i.i.d. random variables which modified by the least-squares-best predictor tends toward the $G$-normal distribution. Corollary \ref{cltv} indicates the broad applicability of $G$-normal distribution. In particular, for i.i.d. random variables with mean-certainty, the corresponding least-square-best predictor are constants which equals to their certain mean. That is just the CLT in \cite{P2019} for the certainty mean. }}

We also have the CLT in the sense of capacity.
\begin{corollary}\label{coro}
Under the same assumptions of Theorem \ref{fcltsub}, we have
\begin{equation}\label{cap}
\lim_{n\rightarrow\infty}\sup_{P\in\cp_*^{(n)}}P\lt(a\leq\frac{\sum_{i=1}^n(X_i-E_P[X_i|\cf_{i-1}])}{\sqrt{n}}\leq b\rt)=V_\Theta([a,b]),
\end{equation}
where $V_\Theta(A):=\sup_{P\in\cp_\Theta}P(A), \ \forall A\in\cb(\br)$ with $\Theta=[\underline{V}(X_1),\overline{V}(X_1)]$ and $-\infty\leq a<b\leq\infty$.
\end{corollary}
\begin{proof}
For any $0<\ve<\frac{b-a}{2}$, there exists $f^\ve,g^\ve\in C_{b.Lip}(\br)$ such that
$${\bf{1}}_{[a+\ve,b-\ve]}(x)\leq g^\ve(x)\leq {\bf{1}}_{[a,b]}(x)\leq f^\ve(x)\leq {\bf{1}}_{[a-\ve,b+\ve]}(x).$$
Then we have
\begin{align*}\eg[g^\ve(\xi)]&=\liminf_{n\rightarrow\infty}\sup_{P\in\cp_*^{(n)}}E_P\lt[g^\ve\lt(\frac{\sum_{i=1}^nX_i-E_P[X_i|\cf_{i-1}]}{\sqrt{n}}\rt)\rt]\\
&\leq\liminf_{n\rightarrow\infty}\sup_{P\in\cp_*^{(n)}}P\lt(a\leq\frac{\sum_{i=1}^nX_i-E_P[X_i|\cf_{i-1}]}{\sqrt{n}}\leq b\rt)\\
&\leq \limsup_{n\rightarrow\infty}\sup_{P\in\cp_*^{(n)}}P\lt(a\leq\frac{\sum_{i=1}^nX_i-E_P[X_i|\cf_{i-1}]}{\sqrt{n}}\leq b\rt)\\
&\leq\limsup_{n\rightarrow\infty}\sup_{P\in\cp_*^{(n)}}E_P\lt[f^\ve\lt(\frac{\sum_{i=1}^nX_i-E_P[X_i|\cf_{i-1}]}{\sqrt{n}}\rt)\rt]=\eg[f^\ve(\xi)]
\end{align*}
We note that
$$V_\Theta([a+\ve,b-\ve])\leq\eg[g^\ve(\xi)]\leq\eg[f^\ve(\xi)]\leq V_\Theta([a-\ve,b+\ve]),$$
and for each $x\in\br$, by Corollary 3.8 in Hu et al. \cite{HWZ}, $V_\Theta([x-\ve,x+\ve])\rightarrow 0$ as $\ve\rightarrow 0$, which implies (\ref{cap}).
\end{proof}

\section{Examples}

The following example is motivated by Teran \cite{Teran}.
\begin{example}\label{ex1}
Let $(\Omega,\cb(\Omega))$ be a measurable space with $\Omega=0\cup\bn$. Let $X_n(\omega)$ be the $n$-th bit in the binary representation of $\omega$, i.e.,
$$\omega=\sum_{n=1}^\infty 2^{n-1}X_n(\omega)$$
with $X_n(\omega)\in\{0,1\}$.

There does not exist a probability measure $P$ on $\Omega$ such that $\{X_i\}_{i\in\bn}$ is a classical i.i.d. sequence with $E_P[X_1]=\frac{1}{2}$. But we can construct linear expectation space $(\Omega,\mathcal{H},\be)$ such that $\{X_i\}_{i\in\bn}$ is i.i.d. under $\be$ with $\be[X_1]=\frac{1}{2}$.

In particular, by Corollary \ref{cltv}, we have
$$\be\lt[\vp\lt(\frac{\sum_{i=1}^n\lt(X_i-\frac{1}{2}\rt)}{\sqrt{n}}\rt)\rt]=\mathbb{E}_G[\vp(\xi)], \ \ \ \forall \vp\in C_{b.Lip}(\br),$$
where $\xi\sim\mathcal{N}(0,\frac{1}{4})$.

\end{example}

Otherwise, if such $P$ exists, then by classical strong law of large numbers, we have
 $$P\lt(\lim_{n\to\infty}\frac{S_n}{n}=\frac{1}{2}\rt)=1,$$
 but for each $\omega\in\Omega$, $X_n(\omega)=0$ when $n>[\log_2\omega]+2$, thus
 $$\lim_{n\to\infty}\frac{S_n(\Omega)}{n}=0, \ \ \forall \omega\in\Omega,$$
 which is a contradiction.

But for fixed $n\in\bn$, let $P^{(n)}$ be the probability measure on $\Omega$ such that $P^{(n)}(\{i\})=\frac{1}{2^n},\ \forall\ 0\leq i\leq 2^n-1$, and $P^{(n)}(\{i\})=0$, if $i\geq 2^n$. It is clear that $\{X_i\}_{1\leq i\leq n}$ is a classical i.i.d. sequence with $E_{P^{(n)}}[X_1]=\frac{1}{2}$.

 Now we construct $(\Omega,\mathcal{H},\be)$ such that $\{X_i\}_{i\in\bn}$ is i.i.d. as in  Definition \ref{id} and \ref{new-de1}. Indeed, let
 $$\mathcal{H}=\{\vp(X_1,\cdots,X_n), \forall n\in\bn, \forall \vp\in C_{Lip}(\br^n)\},$$
 the linear functional $\be$ on $\mathcal{H}$ is defined by
$$\be[\vp(X_1,\cdots,X_n)]:=E_{P^{(n)}}[\vp(X_1,\cdots,X_n)],\ \ \forall n\in\bn.$$
It can be verified that
$$\be[\vp(X_i)]=\frac{1}{2}(\vp(0)+\vp(1)), \ \forall \vp\in C_{b.Lip}(\br), \ \forall i\in\bn,$$
and for each $n\geq 2$ and $\vp\in C_{b.Lip}(\br^n)$,
$$\be[\vp(X_1,\cdots,X_n)]=\be[\be[\vp(x_1,\cdots,x_{n-1},X_n)]_{(x_1,\cdots,x_{n-1})=(X_1,\cdots,X_{n-1})}].$$
We also note that, for fixed $\omega\in\Omega$, $X_n(\omega)\downarrow 0$, but $\be[X_n]=\frac{1}{2}\nrightarrow 0$.  $\be$ is not regular. In fact, if the linear functional $\be$ satisfying (a) and (b) in Definition \ref{sub} is regular, then by the Daniell-Stone theorem, there exists probability measure which is represented by such $\be$.

We then extend above example to the sublinear expectation space.

\begin{example}
Let $\cp$ be the set of all probability measures on $(\Omega,\cb(\Omega))$ with $\Omega=0\cup\bn$ and $\be[\cdot]=\sup_{P\in\cp}E_P[\cdot]$.

We can verify that the random variables $\{X_i\}_{i\in\bn}$, defined by
$$\omega=\sum_{n=1}^\infty X_n(\omega)2^{n-1}, \ \ \ \forall \omega\in\Omega,$$
are i.i.d. under $\be$ (see Guo et al. \cite{gll}).

It is clear that
$$\be[X_1]=1, \ \ -\be[-X_1]=0, \ \ \overline{V}(X_1)=\frac{1}{4}, \ \ \underline{V}(X_1)=0.$$
Furthermore, $\be$ is not regular since $\be[X_n]=1$ but $X_n(\omega)\downarrow 0$, $\forall \omega\in\Omega$.

By Corollary \ref{cltv}, for each continuous function $\vp\in C(\br)$ with linear growth,
$$\lim_{n\rightarrow\infty}\sup_{P\in\cp}E_P\lt[\vp\lt(\frac{\sum_{i=1}^n(X_i-E_P[X_i|\cf_{i-1}])}{\sqrt{n}}\rt)\rt]=\eg[\vp(\xi)],$$
where $\xi\sim\mathcal{N}(0,[0,\frac{1}{4}])$.
\end{example}

%

In the end, the following example shows that condition (\ref{uc}) in CLT can not weaken to $\be[|X_1|^2]<\infty$.
\begin{example}
Let $\Omega = \mathbb{Z}$, $\mathcal{F}=\mathcal{B}(\mathbb{Z})$, $\mathcal{P}=\{P_{k}, k \geq 1\}$, where $P_{k}(\{0\})=1-\frac{1}{k^2}$, $P_{k}(\{k\})=P_{k}(\{-k\})=\frac{1}{2k^2}$. Consider a function $X$ on $\mathbb{Z}$ defined by $X(\omega)=\omega,  \omega\in \mathbb{Z}$ and the sublinear expectation $\be[\cdot]=\sup_{P\in\cp}E_P[\cdot]$. We note that $\be[X]=\be[-X]=0$ and $\be[X^2]=-\be[-X^2]=1$. We are able to construct an i.i.d. sequence $\{X_{i}\}_{i\in\bn}$ on some sublinear expectation space $(\tilde{\Omega},\tilde{\mathcal{H}},\tilde{\mathbb{E}})$ such that $X_{i}$ has the same distribution with $X$. Then CLT does not hold.
\end{example}
In fact, we consider $\varphi(x)=1-|x|$. Let $S_n=X_1+\cdots+X_n$.

We first observe that
\begin{align*}\tilde{\mathbb{E}}\lt[\vp\lt(\frac{x+X}{\sqrt{n}}\rt)\rt]&=\sup_{k\in\mathbb{N}}\lt\{\lt(1-\frac{1}{k^2}\rt)\vp\lt(\frac{x}{\sqrt{n}}\rt)+\frac{1}{2k^2}\lt(\vp\lt(\frac{x-k}{\sqrt{n}}\rt)+\vp\lt(\frac{x+k}{\sqrt{n}}\rt)\rt)\rt\}\\
&=\vp\lt(\frac{x}{\sqrt{n}}\rt)+\frac{1}{\sqrt{n}}\sup_{k\in\mathbb{N}}\lt\{\frac{2|x|-|x+k|-|x-k|}{2k^2}\rt\}=\vp\lt(\frac{x}{\sqrt{n}}\rt).
\end{align*}
Then
\begin{align*}
\tilde{\mathbb{E}}\lt[\vp\lt(\frac{S_n}{\sqrt{n}}\rt)\rt]&=\tilde{\mathbb{E}}\lt[\tilde{\mathbb{E}}\lt[\vp\lt(\frac{x+X_n}{\sqrt{n}}\rt)\rt]|_{x=X_1+\cdots+X_{n-1}}\rt]
=\tilde{\mathbb{E}}\lt[\vp\lt(\frac{S_{n-1}}{\sqrt{n}}\rt)\rt]\\
&=\cdots=\tilde{\mathbb{E}}\lt[\vp\lt(\frac{X_1}{\sqrt{n}}\rt)\rt]=\vp(0)=1.
\end{align*}

Finally,
$$\lim_{n\rightarrow\infty}\tilde{\mathbb{E}}\lt[\varphi\lt(\frac{S_n}{\sqrt{n}}\rt)\rt]=1>\eg[\vp(\xi)],$$
where $\xi\sim\mathcal{N}(0,1)$.

\section{Appendix: Law of large numbers under sublinear expectation}

In this section, we firstly consider the law of large numbers on canonical space.

\begin{theorem}
Let $\{X_i\}_{i\in\bn}$ be a canonical i.i.d. sequence on $(\br^\bn,\cb(\br^\bn))$ and $\cp$ be the set of probability measures on $\br^\bn$ construct by (\ref{e1}) with $\cp_i=\cp_1$, where $\cp_1$ satisfies the following condition
\begin{equation}\label{uccc}
\lim_{\lambda\to\infty}\mathbb{E}^{\cp_1}[(|X_1|-\lambda)^+]=0.
\end{equation}
Then for each $\vp\in C(\br)$ with linear growth, we have
\begin{equation}\label{lln1}
\lim_{n\to\infty}\pe\lt[\vp\lt(\frac{\sum_{i=1}^nX_i}{n}\rt)\rt]=\max_{\lu\leq\mu\leq\ou}\vp(\mu),
\end{equation}
where $\ou=\pe[X_1]$ and $\lu=-\pe[-X_1]$.

In particular,
\begin{equation}\label{lln2}
\lim_{n\to\infty}\pe\lt[d_\Theta\lt(\frac{\sum_{i=1}^nX_i}{n}\rt)\rt]=0,
\end{equation}
where $d_\Theta(x)=\inf_{y\in\Theta}|x-y|$ with $\Theta=[\lu,\ou]$.
\end{theorem}

\begin{proof}
It is clear that (\ref{lln2}) can be implied from (\ref{lln1}) by taking $\vp(x)=d_\Theta(x)$. To prove (\ref{lln1}), we only need to prove that it holds for $\vp\in C_{b,Lip}(\br)$. Indeed, since
\begin{align*}
\pe\lt[\lt(\lt|\frac{S_n}{n}\rt|-\lambda\rt)^+\rt]\leq\frac{1}{n}\pe[(\sum_{i=1}^n|X_i|-\lambda)^+]\leq\pe[(|X_1|-\lambda)^+]\to 0,
\end{align*}
we can extend (\ref{lln1}) from $C_{b,Lip}(\br)$ to $C(\br)$ with linear growth by the similar argument in Peng \cite{P2019}.

The proof of (\ref{lln1}) for $\vp\in C_{b.Lip}(\br)$ is divided into two steps.

Firstly, let $\tX_k=(-n\vee X_k)\wedge n$, $1\leq k\leq n$, $\ou_n=\pe[\tX_n]$ and $\lu_n=-\pe[-\tX_n]$, $\forall n\in\bn$. We denote $\{\cf_i\}_{i\in\bn}$ the natural filtration generalized by $\{X_i\}_{i\in\bn}$ and  $S_n=\sum_{i=1}^nX_i$, $\tilde{S}_n=\sum_{i=1}^n\tX_i$.

For each $P\in\cp$, we have
$$E_P\lt[\lt|\vp\lt(\frac{S_n}{n}\rt)-\vp\lt(\frac{\tilde{S}_n}{n}\rt)\rt|\rt]\leq \frac{L_\vp}{n}E_P[|\sum_{i=1}^n(X_i-\tX_i)|]\leq L_\vp\pe[(|X_1|-n)^+]\to 0.$$
\begin{align*}
&E_P\lt[\lt|\vp\lt(\frac{\sum_{i=1}^n\tilde{S}_n}{n}\rt)-\vp\lt(\frac{\sum_{i=1}^nE_P[\tX_i|\cf_{i-1}]}{n}\rt)\rt|^2\rt]\\
\leq& \frac{L_\vp^2}{n^2}E_P[\sum_{i=1}^n(\tX_i-E_P[\tX_i|\cf_{i-1}])^2]\\
\leq&\frac{L_\vp^2\sum_{i=1}^nE_P[(\tX_i-E_P[\tX_i|\cf_{i-1}])^2]}{n^2}\leq \frac{2L_\vp^2\pe[X_1^2]}{n}\to 0
\end{align*}
where $L_\vp$ is Lipschitz constant of $\vp$. The proof of the last convergence is very similar to the classical case.

It is clear that $\ou_n\rightarrow\ou$ and $\lu_n\to\lu$ by (\ref{uccc}). Combining with Proposition \ref{prop23},  we imply that
$$\vp\lt(\frac{\sum_{i=1}^nE_P[\tX_i|\cf_{i-1}]}{n}\rt)\leq\max_{\frac{\sum_{i=1}^n\lu_i}{n}\leq\mu\leq\frac{\sum_{i=1}^n\ou_i}{n}}\vp(\mu)\to\max_{\lu\leq\mu\leq\ou}\vp(\mu).$$

Thus we can conclude that
$$\limsup_{n\to\infty}\pe\lt[\vp\lt(\frac{S_n}{n}\rt)\rt]\leq\max_{\lu\leq\mu\leq\ou}\vp(\mu).$$

Secondly, for fixed $\vp\in C_{b.Lip}(\br)$, there exists $\mu^*\in[\lu,\ou]$ such that $\vp(\mu^*)=\max_{\lu\leq\mu\leq\ou}\vp(\mu)$. By the construction of $\cp$, there exists $P_*\in\cp$ such that $\{X_i\}_{i\in\bn}$ is an i.i.d. sequence with $E_{P^*}[X_1]=\mu^*$. The classical law of large numbers shows that
$$\lim_{n\to\infty}E_{P^*}\lt[\vp\lt(\frac{S_n}{n}\rt)\rt]=\vp(\mu^*),$$
which implies that
$$\liminf_{n\to\infty}\pe\lt[\vp\lt(\frac{S_n}{n}\rt)\rt]\geq\max_{\lu\leq\mu\leq\ou}\vp(\mu).$$
The proof is completed.

\end{proof}

By Proposition \ref{pp1}, we can extend law of large numbers from canonical space to the sublinear expectation space, which is the law of large numbers in Peng \cite{P2019}.
\begin{theorem}
Let $\{X_i\}_{i\in\bn}$ be an i.i.d. sequence on sublinear expectation space $(\Omega,\mathcal{H},\be)$ with
$$\lim_{\lambda\rightarrow\infty}\be[(|X_1|^2-\lambda)^+]=0.$$
For each $\vp\in C(\br)$ with linear growth, we have
$$\lim_{n\to\infty}\be\lt[\vp\lt(\frac{S_n}{n}\rt)\rt]=\max_{\lu\leq\mu\leq\ou}\vp(\mu),$$
where $\lu=-\be[-X_1]$ and $\ou=\be[X_1]$.
\end{theorem}

\section*{Acknowledgements}

The author gratefully acknowledges the many helpful suggestions of Prof. Shige PENG during the preparation of the paper.

This work was supported by NSF of Shandong Provence (No.ZR2021MA018),  NSF of China (No.11601281),  National Key R\&D Program of China (No.2018YFA0703900) and the Young Scholars Program of Shandong University.




\begin{thebibliography}{99}                                                                                               %

\bibitem{brown} B. Brown, Martingale central limit theorems, Ann. Mathe. Statis., 42(1) 1971, 59-66.

\bibitem{CE} Z. Chen, L. Epstein, A central limit theorem for sets of probability measures, arXiv:2006.16875v1, (2020).


\bibitem{CEZ} Z. Chen, L. Epstein, G. Zhang, A central limit theorem, loss aversion and multi-armed bandits, arXiv:2106.05472v1, (2021).


\bibitem {DHP11}L. Denis, M. Hu, S. Peng, Function spaces and capacity related
to a sublinear expectation: application to $G$-Brownian motion paths,
Potential Anal., 34 (2011), 139-161.



\bibitem{DNS} Y. Dolinsky, M. Nutz, M. Soner, Weak approximation of G-expectations, Stoch. Process. Appl., 122 (2012), 664-675.



\bibitem {FPSS}X. Fang, S. Peng, Q. Shao, Y. Song, Limit theorems with rate of
convergence under sublinear expectations, Bernoulli, 25 (2019), 2564-2596.




\bibitem{gll}X. Guo, S. Li and X. Li, On the laws of the iterated logarithm with mean-uncertainty under the sublinear expectations, preprint, (2022).

\bibitem{gll2}X. Guo, S. Li and X. Li, Liapounov's type central limit theorem with model uncertainty and its application in number theory, preprint, (2022).

\bibitem{GL}X. Guo and X. Li, On the laws of large numbers for pseudo-independent random variables under sublinear expectation, Statist. and Probab. Lett., 172, (2021), 109042.



\bibitem {HL}M. Hu, X. Li, Independence Under the $G$-Expectation Framework,
J. Theor. Probab., 27 (2014), 1011-1020.

\bibitem{HLL} M. Hu, X. Li, X. Li, Convergence rate of Peng's law of large numbers under sublinear expectations, Probab. Uncertain. Quant. Risk, 6(3) (2021), 261-266.



\bibitem {HP09}M. Hu, S. Peng, On representation theorem of G-expectations and
paths of $G$-Brownian motion, Acta Math. Appl. Sin. Engl. Ser., 25 (2009), 539-546.


\bibitem{HWZ} M. Hu, F. Wang, G. Zhang, Quasi-continuous random variables and processes under the $G$-expectation framework, Stochastic Process. Appl., 126(8) (2016), 2367-2387.

\bibitem{Kry2} N. Krylov, Nonlinear parabolic and elliptic equations of the second order. Translated from the Russian, Dordrecht (1987).

\bibitem{Kry1} N. Krylov, On Shige Peng's central limit theorem, Stochastic Process. Appl., 130 (2020), 1426-1434.


\bibitem{KP} T. Kurtz, P. Protter, Weak limit theorems for stochastic integrals and stochastic differential equations, Ann. Prob. 19(3) (1991), 1035-1070.

\bibitem{LS} M. Li and Y. Shi, A general central limit theorem under sublinear expectations. Sci. China Math., 53(8), (2010), 1989-1994.


\bibitem{lly} S. Li, X. Li and X. Yuan, Upper and lower variances under model uncertainty and their applications in finance, to appear in International Jrounal of Finanical Engineering, (2022).


\bibitem{Li2015} X. Li, A central limit theorem for $m$-dependent random
variables under sublinear expectations, Acta Math. Appl. Sin. Engl. Ser., 31(2) (2015), 435-444.

\bibitem{LL}X. Li and Y. Lin, Generalized Wasserstein distance and weak
convergence of sublinear expectations, J. Theor. Probab., 30 (2017), 581-593.



\bibitem{PP} E. Pardoux, S. Peng, Adapted solution of a backward stochastic differential equation, Sys. Control Lett. 14, (1990), 55-61.

\bibitem{P97} S. Peng, Backward SDE and related $g$-expectation, In N. El Karoui and L. Mazliak eds. Backward Stochatic Differential Equations, Pitman Reserch Notes in Math. Series 364, 141-159. (1997)


\bibitem{P08} S. Peng, A new central limit theorem under sublinear expectations, (2008), arXiv:0803.2656.



\bibitem {pengsur}S. Peng, Survey on normal distributions, central limit
theorem, Brownian motion and the related stochastic calculus under sublinear
expectations, Science in China, Series A. Mathematics, 52 (2009), 1391-1411.


\bibitem{P09} S. Peng, Tightness, weak compactness of nonlinear expectations and application to CLT, (2010), arXiv:1006.2541.

\bibitem{P10} S. Peng, Backward Stochastic Differential
Equation, Nonlinear Expectation and
Their Applications, (2010), Proceedings of the International Congress of Mathematicians
Hyderabad, India

\bibitem {P2019}S. Peng, Nonlinear Expectations and Stochastic Calculus under
Uncertainty, Springer (2019).

\bibitem {P2019b}S. Peng, Law of large numbers and central limit theorem under
nonlinear expectations, Probab. Uncertain. Quant. Risk, 4 (2019), 1-8.



\bibitem{sion} M. Sion,  On general minimax theorems, Pacific J. Math., 8 (1958), 171-176.


\bibitem{song2} Y. Song, Normal approximation by Stein's method under sublinear expectations, Stochastic Process. Appl., 130 (2020), 2838-2850.

\bibitem{Teran} P. Ter\'{a}n, Sublinear expectations: On large sample behaviours, Monte Carlo method, and coherent upper previsions. In E. Gil et al. (eds.) The Mathematics of the Uncertain, Studies in Systems, Decision and Control 142, 2018.

\bibitem{walley} P. Walley, Statistical Reasoning with Imprecise Probabilities, Chapman \& Hall (1991).

\bibitem{yan} J. Yan, Measure Theory (3rd Edition)(in Chineses), Science Press, (2021).


\bibitem{Z2015} L. Zhang, Donsker's invariance principle under the sub-linear expectation with an application to Chung's law of the itertated logarithm, Commun. Math. Stat. 3, (2015), 187-214.

\bibitem{Zhang} L. Zhang, The convergence of the sums of independent random variables under the sub-linear expectations, Acta Math. Appl. Sin. Engl. Ser., 36(3), (2020), 224-244.



\bibitem{zhang2020} L. Zhang, Lingdeberg's central limit theorems for martingale-like sequences under sub-linear expectations, Science China Mathematics, 64(6), (2021), 1263-1290.
\end{thebibliography}
\end{document}